\documentclass[a4paper,12pt]{article}

\usepackage[top=1in,left=1in,right=1in,bottom=1.5in]{geometry}

\usepackage[T1]{fontenc}
\usepackage{amsmath,amssymb,amsfonts,amsthm,mathrsfs,eucal,dsfont,verbatim}

\numberwithin{equation}{section}
\theoremstyle{plain}
\newtheorem{theorem}{Theorem}
\newtheorem{lemma}[theorem]{Lemma}

\theoremstyle{definition}
\newtheorem{remark}[theorem]{Remark}
\newtheorem{definition}{Definition}
\newtheorem{example}[theorem]{Example}

\numberwithin{theorem}{section}
\numberwithin{definition}{section}

\newcommand{\eps}{\varepsilon}
\newcommand{\prt}{\partial}

\newcommand{\real}{\mathds{R}}
\newcommand{\rn}{{\mathds{R}^n}}
\newcommand{\comp}{\mathds{C}}

\newcommand{\I}{\mathds 1}

\def\1{1\!\!\hbox{{\rm I}}}
\renewcommand{\Re}{\mathbb{R}}
\newcommand{\be}{\begin{equation}}
\newcommand{\ee}{\end{equation}}
\newcommand{\ba}{\begin{aligned}}
\newcommand{\ea}{\end{aligned}}
\newcommand{\RE}{\mathop{\mathrm{Re}}}

\usepackage{color}

\evensidemargin
-1cm

\title{Intrinsic small time estimates for distribution densities of L\'evy processes}

\author{%
    \textsc{Victoria Knopova}%
    \thanks{V.M.\ Glushkov Institute of Cybernetics,
            NAS of Ukraine,
            40, Acad.\ Glushkov Ave.,
            03187, Kiev, Ukraine,
            \texttt{vic\underline{ }knopova@gmx.de}}
    \textrm{\ \ and\ \ }
    \stepcounter{footnote}\stepcounter{footnote}\stepcounter{footnote}
    \stepcounter{footnote}\stepcounter{footnote}%
    \textsc{Alexey Kulik}%
    \thanks{Institute of Mathematics, NAS of Ukraine, 3, Tereshchenkivska str., 01601  Kiev, Ukraine,
    \texttt{kulik@imath.kiev.ua}}
    }

\date{}

\begin{document}

\maketitle

\begin{abstract}
    \noindent
    We  construct intrinsic on-and off-diagonal upper and lower estimates for the transition probability density of a L\'evy process in small time.
    By intrinsic we mean that such estimates reflect the structure of the characteristic exponent of the process. The technique used in the paper relies on the asymptotic analysis of the inverse Fourier  transform of the respective characteristic function.  We provide several examples, in particular, with rather irregular L\'evy measure, to illustrate our results.


\end{abstract}

\section{Introduction}

A wide variety of asymptotic results on the small time behaviour of the transition probability density for  L\'evy and, more generally, L\'evy driven and L\'evy-type Markov processes is available in the literature. In \cite{Le87} and  \cite{Ish94} this problem was treated by means of the Malliavin calculus, see also   \cite{Pi97a}, \cite{Pi97b},  \cite{Ish01}, where the Picard's modification of the Malliavin calculus introduced in \cite{Pi96}  was used. Another approach, applicable for L\'evy-type Markov processes and   developed in  \cite{BBCK09},
 \cite{CKK11}, \cite{C09},
 \cite{BGK09},
\cite{CK03},  \cite{CK08},  \cite{CKK09}, is based essentially on  the method introduced in \cite{CKS87}, and involves the Dirichlet form technique.  We refer also to more specific papers \cite{St10a}, \cite{St10b}, \cite{RS10}, where the particular structure of tempered  stable processes is used, and to  \cite{BN00}, \cite{FO11}, \cite{FH09}, \cite{RW02} for other results on the small time asymptotic expansions. Of course, this list of publications is far from being complete.

Typically, the results available in the field involve some restrictions on the initial L\'evy noise, formulated in an explicit and prescribed form,  e.g.,  the two-sided power bounds for the truncated second moment of the L\'evy measure in \cite{Pi97a}, \cite{Pi97b}, and  \cite{Ish01}; see (\ref{orey}) below.  On the contrary, the  approach developed in the current paper is free from any assumption of such a type.  We introduce two auxiliary functions $\psi^L$ and $\psi^U$, see (\ref{psiL}) and (\ref{psiU}) below, which give natural lower and upper bounds for the real part of the characteristic exponent $\psi$ of the L\'evy process, see (\ref{psipm1}) below. Our principal assumption is that the trivial inequality $\psi^L\leq \psi^U$ is invertible, in a sense; see (\ref{A}). This assumption is satisfied for a rather wide range of L\'evy processes  including, for instance, the  well studied stable and ``$\alpha$-stable like'' processes, see Examples~\ref{exa1} and \ref{exa2}. On the other hand, this assumption may hold true as well  even when  the real part of the characteristic exponent of the process  exhibits a rather irregular asymptotic behavior at $\infty$, see Example~\ref{exa3}.

Under assumption (\ref{A}) we are able to give  two  types of small time estimates for the distribution density $p_t(x)$ of the process $Z_t$. First,  we prove the  \emph{on-diagonal estimates}, i.e. two-sided bounds for the  maximal value of the distribution density: for given $t_0>0$  there exist positive constants $c$ and $d$ such that
\be\label{unif}
c\rho_t\leq \max_{x\in \Re} p_t(x)\leq d\rho_t, \quad t\in (0, t_0],
\ee
with $\rho_t=\big(\RE\psi\big)^{-1}\big(1/t\big)$,
where
$\big(\RE\psi\big)^{-1}(x):=\inf\{\zeta>0: \RE\psi(\zeta)\geq x\}$, $x>0$,
is the quasi-inverse function of the real part $\RE\psi$ of the characteristic exponent.  Second, we give more informative, but more complicated, two-sided {\it off-diagonal estimates}. The structure of such estimates is a subject of separate discussion.  It is well known (e.g. \cite{Ish94}, \cite{Pi97a}, \cite{Pi97b},  \cite{Ish01}),  that for the transition probability density of a L\'evy driven  Markov process such a  structure, in general,  should be more complicated than a ``bell-like'' estimate, which is typical for the diffusion case. We discuss this feature in details in Section~2.2 and in Examples~\ref{exa1}, \ref{exa2} below. Motivated by this observation,  we propose a new type of non-uniform estimates,  which we call \emph{compound kernel estimates}. For both types of the estimates described above we use the term \emph{intrinsic}, because the structure of the estimates is mainly determined by the characteristic exponent and L\'evy measure of the process. In the  main results any  prescribed function, e.g. power-type, regularly varying, etc. is not involved neither in the assumptions, nor in the estimates themselves. In addition, we discuss some auxiliary assumptions which make it possible to  simplify our main estimates, and to present them  in a more compact form.

A key ingredient in our approach is the asymptotic analysis of the inverse integral Fourier transform of the characteristic function of the  initial process, which involves a proper version of the Laplace method (e.g. \cite{Co65}), and is a ``small-time counterpart'' of the asymptotic analysis developed in \cite{KK11} and \cite{K12}. For such an analysis, it is important to have the respective characteristic function  explicitly. Consequently,  the range of applications of such an approach seemingly does not exceed the class of L\'evy processes (or, slightly more general, L\'evy driven stochastic integrals with deterministic kernels). The next natural step in this research direction would be an extension of the asymptotic results obtained in this paper to L\'evy driven and L\'evy-type Markov processes using  a perturbation technique, analogous to the \emph{parametrix method} for the diffusion processes. This is a subject of a separate forthcoming paper \cite{KK12}.

To make the exposition as transparent as possible, we formulate our main estimates for   one-dimensional  L\'evy processes. Multidimensional generalization seemingly would not contain any substantially new difficulties. We postpone such a generalization to further publications.

 We now outline the rest of the paper.  In Section~\ref{Set} we introduce the notation and  the assumptions,  and formulate our main results.
Proofs are given in Section~\ref{proofs}. In Section~\ref{exa} we introduce and discuss some examples.

\section{Main results}\label{Set}
 \subsection{Notation, assumptions, and on-diagonal  bounds}  In the sequel,  $Z_t$ is a real-valued L\'evy process with the characteristic exponent $\psi$, i.e.
$$
Ee^{i \xi Z_t} = e^{-t\psi(\xi)},\quad \xi\in \Re.
$$
We assume that $Z_t$ does not contain a Gaussian component. In this case $\psi$ admits the L\'evy-Khinchin representation
\begin{equation}
    \psi(\xi)=ia\xi +\int_\real (1-e^{i\xi u}+i\xi u\I_{|u|<1})\mu(du), \label{psi}
    \end{equation}
where $a\in \Re$ and $\mu$ is a L\'evy measure, i.e. $\int_\real (1\wedge u^2)\mu(du)<\infty$.  We will treat separately the situation where  the process $Z_t$ is symmetric, which means  in terms of the characteristic exponent  that   $\psi$ is real-valued.  In what follows we assume
\be\label{mu}
\mu(\Re)=\infty,
\ee
which is clearly necessary for $Z_t$ to possess a distribution density.

Let
 \be\label{psiL}
L(x):=x^2\I_{|x|<1}, \quad  \psi^L(\xi):=\int_{\Re}L(\xi u)\mu(du)=\int_{|\xi u|<1} (\xi u)^2\mu(du),
\ee
\be\label{psiU}
U(x):=x^2\wedge 1,\quad  \psi^U(\xi):=\int_{\Re}U(\xi u)\mu(du)=\int_{\real} \left((\xi u)^2\wedge 1\right)\mu(du).
\ee

\noindent Note that  up to a constant multiplier the functions $\psi^L$ and $\psi^U$ provide, respectively, a lower and an upper bound for the real part of  the characteristic exponent
$$
\RE \psi(\xi)=\int_\Re\Big(1-\cos(\xi u)\Big)\mu(du).
$$
Indeed,  from the elementary inequalities
   $$
    (1-\cos 1) \leq \frac{1-\cos x}{x^2} \leq \frac{1}{2}, \quad |x|<1,\quad \text{and } \quad  0\leq 1-\cos x\leq 2, \quad x\in \Re,
    $$
 we have
    \begin{equation}
    (1-\cos 1) \psi^L(\xi)\leq \RE\psi(\xi) \leq 2 \psi^U(\xi), \quad \xi\in \Re. \label{psipm1}
    \end{equation}

Our main assumption  formulated  below states that  these bounds are in a sense exact.

\medskip

\begin{itemize}

\item[\textbf{A.}] There exists some $\beta>1$ such that
\be\label{A}\psi^U(\xi)\leq \beta\psi^L(\xi), \quad \xi\in \Re.
\ee

\end{itemize}

 For  $t> 0$, we  put
\begin{equation}
    \rho_t :=\big(\RE\psi\big)^{-1}\left(1/t\right)=\inf\{ \xi> 0:  \RE\psi(\xi)\geq 1/t\}.\label{rho}
\end{equation}
Note that under condition \textbf{A}  the set in the right hand side of \eqref{rho} is not empty (cf.  Lemma~\ref{C1} below).

 \begin{theorem}\label{t-main} Suppose that condition \textbf{A}   is satisfied.  Then for every $t>0$ the distribution  density $p_t$ of $Z_t$ exists, is infinitely differentiable, and  for every $t_0>0$ there exist  positive constants $c\equiv c(t_0), d\equiv d(t_0)$ such that (\ref{unif}) holds true.

\end{theorem}

\emph{Remark.} We will see in the proof below that $p_t(\cdot)$ vanishes at $\pm\infty$, which means that the maximum in  (\ref{unif}) is well defined.

\subsection{Off-diagonal  bounds}

Apart from the on-diagonal bounds of the form (\ref{unif}), one is typically  interested in  the off-diagonal bounds, which would control the behavior of $p_t(x)$ w.r.t. the space variable $x$. Below we propose a new type of such  bounds, which we call \emph{compound kernel estimates}. Let us discuss this topic in details.

A typical form of an upper estimate  one may think about, is the ``bell-like'' estimate
\be\label{bell}p_t(x)\leq f_t(x)
\ee
with some kernel $f_t(x)$ which, when considered as a function of $x$ for a fixed $t>0$, is symmetric on $\Re$ and decreasing  on $\Re^+$. In addition, one would expect  the kernel $f_t(x)$
to have the form
\be\label{bell_scale} f_t(x)= \sigma_t f(\sigma_t x),
\ee
where the ``shape function''  $f(x)$ is re-scaled by some   ``scale function'' $\sigma_t$. Typical examples which give  a strong motivation for
the estimate (\ref{bell}) with $f_t(x)$ of the form (\ref{bell_scale}), are the Brownian motion and, more generally, the $\alpha$-stable process. In this case,   \eqref{bell} and \eqref{bell_scale} hold true with $\sigma_t=1/\sqrt{t}$ and $\sigma_t=t^{-1/\alpha}$, respectively. However, for more general L\'evy processes it may happen that any upper estimate of the form (\ref{bell}) with $f_t(x)$ as in  (\ref{bell_scale}), is very inexact.  To see that, we recall the papers
\cite{Ish94}, \cite{Pi97a}, \cite{Pi97b},  \cite{Ish01}, where the pointwise asymptotic behavior of the transition probability densities of  L\'evy driven Markov processes is studied. In our setting, the results from  \cite{Ish94}, \cite{Pi97a}, \cite{Pi97b},  \cite{Ish01} imply  that for some (not necessarily all) points $x\in \Re$ and for every $t_0>0$ there exist positive constants $c=c(x,t_0),d=d(x, t_0)$  such that
$$
ct^{\gamma(x)}\leq p_t(x)\leq d t^{\gamma(x)}, \quad t\in (0,t_0].
$$
Here the power $\gamma(x)$ may vary from point to point, which does not correspond well to the ``bell-like'' estimates (\ref{bell}), (\ref{bell_scale}). In Example \ref{exa2} below we extend this observation and give an example of the L\'evy process such that, if  one constructs any estimate of the form (\ref{bell}) with   the function $f_t(x)$ monotonous on $\Re^+$ with respect to  $x$, then this function $f_t(x)$ is necessarily non-integrable on $\Re$ for every $t>0$.
Since
$$
\int_\Re p_t(x)\, dx=1, \quad t>0,
$$
we see that  for this L\'evy process any estimate of the form (\ref{bell}) with a ``bell-like function'' $f_t(x)$ (not necessarily of the form (\ref{bell_scale})) is extremely inexact. Note that in Example \ref{exa2}, as well as in  \cite{Ish94}, \cite{Pi97a}, \cite{Pi97b},  \cite{Ish01}, the L\'evy measure of the noise is \emph{$\alpha$-stable like} in the sense that
\be\label{orey}
c\eps^{2-\alpha}\leq \int_{|u|\leq\eps} u^2\mu(du)\leq d\eps^{2-\alpha}, \quad \eps\in (0,1].
\ee
This means that even in the  case, which is very close to the $\alpha$-stable one, in general one can not expect  to obtain a sufficiently exact upper estimate of the form (\ref{bell}) with a ``bell-like function'' $f_t(x)$ without further restrictions on the L\'evy measure of the noise.

To explain the proper structure of an upper estimate, we represent $Z_t$ as the sum
\be\label{decomp}
Z_t=\bar Z_t+\hat Z_t-a_t,
\ee
where $\bar Z_t$ and $\hat Z_t$ are  independent,
 $\bar Z_t$ possessing for each $t>0$  the  infinitely divisible distribution with the characteristic exponent
\begin{equation}
    \psi_t(\xi)=t\int_{|\rho_t u|\leq 1} (1-e^{i\xi u}+i\xi u)\mu(du), \label{psit}
    \end{equation}
$\hat Z_t$ having the compound Poisson distribution with the intensity measure
\begin{equation}
\Lambda_t(du)=t  \mu(du)1_{| \rho_t u|>1}, \label{lam}
\end{equation}
and $$
a_t=t\left(a+\int_{1/\rho_t<|u|<1}u\, \mu(du)\right),
$$
 where $\rho_t$ is defined in (\ref{rho}).  We  show below that if  condition \textbf{A} is satisfied, then  $\bar Z_t$ has a distribution density $\bar p_t(x)$ which satisfies (\ref{bell}) and  (\ref{bell_scale}) with $\sigma_t=\rho_t$. On the other hand,
\begin{equation}
p_t(x)=(\bar p_t* P_t*\delta_{-a_t})(x)\label{conv1}
\end{equation}
with
\begin{equation}
    P_t(dy)=e^{-\Lambda_t(\real)} \sum_{m=0}^\infty \frac{1}{m!} \Lambda_t^{*m}(dy),\label{Poist}
    \end{equation}
where $\Lambda_t^{*m}$ denotes the $m$-fold convolution power of the measure $\Lambda_t$, and $\Lambda_t^{*0}$ is equal to the $\delta$-measure at $0$. This observation leads to an upper estimate for  $p_t(x+a_t)$ in the form of the series of convolutions of the single kernel function $\bar f_t(x)$ involved into  (\ref{bell}) with  convolution powers of the measure $\Lambda_t$. The above  construction motivates the following general definition.

\begin{definition}
Let $\sigma, \zeta: (0, \infty)\to \Re$, $h:\Re\to \Re$ be some functions, and $(Q_t)_{t\geq 0}$ be a family of finite measures on the Borel $\sigma$-algebra in $\Re$.  We say that a real-valued function $g$ defined on a set $A\subset (0, \infty)\times \Re$ satisfies the \emph{upper compound kernel estimate} with parameters $(\sigma_t, h, \zeta_t, Q_t)$, if
\be\label{upperCCE}
g_t(x)\leq\sum_{m=0}{1\over m!}\int_{\Re}\sigma_t  h((x-y)\zeta_t)Q_t^{*m}(dy), \quad (t,x)\in A.
\ee
 If the analogue of \eqref{upperCCE} holds true with the sign $\geq$ instead of $\leq$, then we say that the function $g$ satisfies the \emph{lower  compound kernel estimate} with parameters $(\sigma_t, h, \zeta_t, Q_t)$.
\end{definition}
In the above definition, the function $h$ heuristically  determines the shape of a ``single principal kernel'' involved in the estimate, which is re-scaled by the functions $\sigma_t$ and $\zeta_t$ and then extended in a ``compound Poisson'' like way by the family of  ``transition measures'' $Q_t$, $t>0$.

Now we are ready to formulate the main result of this section. Denote
$$
x_t:=\min\left\{x:\quad \bar p_t(x)=\max_{x'\in \Re}\bar p_t(x')\right\}.
$$
Below we will see that $x_t$ is well defined, and for every $t_0>0$ there exists $L=L(t_0)$ such that
$$
|x_t|\leq L/\rho_t,\quad  t\in (0, t_0].
$$

 \begin{theorem}\label{t-main1} Let condition \textbf{A}   hold true.  Then for every $t_0>0$ there exist constants $b_1, b_2, b_3, b_4$, such that the following holds true:
 \begin{itemize}\item[I.]   The function
$$
 p_t(x+a_t), \quad (t, x)\in (0, t_0]\times \Re,
$$
satisfies the upper compound kernel estimate with parameters $(\rho_t, f_{upper}, \rho_t, \Lambda_t)$, where
\begin{equation}
f_{upper}(x)=b_1 e^{-b_2|x|}.\label{fup}
\end{equation}

\item[II.]   The function
$$
 p_t(x+a_t-x_t), \quad (t, x)\in (0, t_0]\times \Re,
$$
satisfies the lower compound kernel estimate with parameters $(\rho_t, f_{lower}, \rho_t, \Lambda_t)$, where
\begin{equation}
f_{lower}(x)=b_3 \I_{|x|\leq b_4}.\label{flower}
\end{equation}

 \end{itemize}
\end{theorem}

Under some additional restrictions on the L\'evy measure $\mu$, an upper compound kernel estimate from Theorem \ref{t-main1} may lead to the ``bell-like'' estimate of the form (\ref{bell}). We formulate two statements of such a type in Theorem \ref{Klup} below. However, as we have mentioned it  in the discussion above, any statement of such a kind would necessarily be more sensitive with respect to the structure of the L\'evy measure of the noise than  that of  Theorem \ref{t-main1}, which involves {\itshape only} our main condition \textbf{A}. This is the reason for us to focus on the   compound kernel estimates. To finalize the discussion on that subject, let us mention that compound kernel estimates preserve some principal features which  the ``bell-like'' estimates exhibit  in the diffusion case. In particular, it is possible to extend compound kernel estimates, obtained in Theorem \ref{t-main1} for L\'evy processes, to the L\'evy-type setting using an analogue of the \emph{parametrix method},  see the forthcoming paper \cite{KK12}. In the same paper we give  some applications of the compound kernel estimates to the investigation of local times for  L\'evy-type Markov processes.

\subsection{Extensions and corollaries}

In this section we formulate some extensions and corollaries of the above results. The first extension gives  compound kernel estimates for the derivatives of the distribution density $p_t(x)$.

 \begin{theorem}\label{aux-theo}
 Let condition \textbf{A}   hold true.
Then for every $k\geq 1$, $t_0>0$,  there exist constants $b_1, b_2$, such that the function
$$
\left|{\partial^k\over \prt x^k} p_t(x+a_t) \right|, \quad (t, x)\in (0, t_0]\times \real,
$$
satisfies the upper compound kernel estimate with parameters  $(\rho_t^{k+1}, f_{upper}, \rho_t, \Lambda_t)$,  where the function $f_{upper}$ is defined in (\ref{fup}).

\end{theorem}

The second extension gives more precise versions of Theorems \ref{t-main1} and \ref{aux-theo}  in the case where the process $Z_t$ is symmetric.

 \begin{theorem}\label{aux-theo-sym}  Let the process $Z_t$ be symmetric and condition \textbf{A}   hold true.
Then the statements of Theorem \ref{t-main1} and Theorem \ref{aux-theo} hold with $a_t$ and $x_t$ replaced by zero, and $f_{upper}$ defined by
\begin{equation}
f_{upper}(x)=b_1 e^{-b_2|x|\ln (1+|x|)} \label{fup-sym}
\end{equation}
instead of (\ref{fup}).
\end{theorem}

Finally, we give a corollary which provides the  ``bell-like'' estimate for the distribution density $p_t(x)$ and its derivatives. To do that, we recall the notion of a sub-exponential distribution, see \cite{EGV79} and  \cite{Kl89}.

\begin{definition}\label{sub} A distribution function  $G$ on $[0,\infty)$ is called  \emph{sub-exponential}, if

(i) for  every $y\in \real$ one has $\lim_{x\to \infty}\frac{1- G(x-y)}{1- G(x)} =1$;

(ii) $ \lim_{x\to \infty} \frac{1- G^{*2}(x)}{1- G(x)}=2$.

A distribution density $g$  on $[0,\infty)$ is called  \emph{sub-exponential}, if it is positive on $[x_0, \infty)$ for some $x_0\geq 0$, and

(i) for  every $y\in \real$ one has $\lim_{x\to \infty}\frac{ g(x-y)}{g(x)} =1$;

(ii) $\lim_{x\to \infty} \frac{g^{*2}(x)}{g(x)}=2$.

\end{definition}
Note that the tail of a sub-exponential distribution (respectively, a sub-exponential density) decays slower than any exponential function, i.e.
$$
\lim_{x\to \infty} V(x)e^{ax}=\infty,
$$
where $V(x)= 1-G(x)$ (resp.,  $V(x)= g(x)$), and  $a>0$ is arbitrary.  This follows from the representation
$$
V(x)= c(x) e^{-\int_{z_0}^x b(y)dy} \quad \text{for}\quad x\geq z_0,
$$
where $z_0=0$ in the case $V(x)=1-G(x)$, and $z_0=x_0$ taken  as in  Definition~\ref{sub} for   $V(x)=g(x)$. The functions  $c: [z_0,\infty)\to (0,\infty)$ and  $b: [z_0,\infty)\to \real$ are  such  that $\lim_{x\to \infty} c(x)=c>0$, $\lim_{x\to\infty} b(x)=0$, cf. \cite{Kl89}.

\begin{theorem}\label{Klup}
Let  condition \textbf{A} hold true.
\begin{itemize}
\item[I.]  Suppose that there exist a  sub-exponential distribution function  $G$ and a constant $C$ such that
\begin{equation}
t\mu\Big(\{u: |\rho_t u|>v\}\Big)\leq C(1-G(v)), \quad v\in [1, \infty), \quad t\in (0,t_0]. \label{dens02}
\end{equation}

Then  for every $t_0>0$ there exist positive constants $b_1,b_2, C_1$, such that
\be\label{bell1}
p_t(x+a_t)\leq C_1  \rho_t \left( f_{upper}(\rho_t x)+ 1-G(|x\rho_t|)\right), \quad x\in \Re, \quad t\in (0, t_0],
\ee
where $f_{upper}$ is defined by \eqref{fup}.

\item[II.]  Suppose that the L\'evy measure $\mu$ possesses the density $m(u)={\mu(du)\over du}$, and there exist a  sub-exponential distribution density $g$ and  positive constants $C, t_0$ such that
     \begin{equation}
m_t(u):= t\rho_t^{-1}m(\rho_t^{-1}u)\leq Cg(|u|), \quad |u|\geq 1, \quad t\in (0,t_0].  \label{dens01}
\end{equation}
Then  there exist positive constants $b_1$, $b_2$, $C_1$, such that
\begin{equation}
p_t(x)\leq C_1 \rho_t \left(f_{upper} (x\rho_t)+ g(|x\rho_t|)\right),\quad x\in \real,\quad   t\in (0, t_0],
\label{bell2}
\end{equation}
where $f_{upper}$ is defined by \eqref{fup}.
\end{itemize}
\end{theorem}

\begin{remark}\label{rem1}
For Theorem \ref{Klup}, the extensions analogous to those given in Theorem \ref{aux-theo} and Theorem \ref{aux-theo-sym} are valid. Namely, for every $k\geq 1$, under the proper choice of the constants $b_1, b_2, C_k$, we have
\be\label{bell_shift_der}
\left|{\partial^k\over \partial x^k}p_t(x+a_t)\right|\leq C_k \rho_t^{k+1} f(\rho_t x), \quad x\in \Re, \quad t\in (0, t_0],
\ee
with
\begin{equation}
f(x)=
\begin{cases}
f_{upper}(x) + 1-G(|x|),& \text{when $\mu$ satisfies \eqref{dens02}},
\\
f_{upper}(x)+ g(|x|), & \text{when $\mu$ satisfies (\ref{dens01})}.
\end{cases}\label{f}
\end{equation}
 For a symmetric process $Z_t$ we have $a_t=0$,  and can take $f_{upper}$  in \eqref{f}   defined by \eqref{fup-sym}.
\end{remark}

\section{Proofs}\label{proofs}

\subsection{Auxiliaries}  We begin  with a brief discussion of the properties of the auxiliary functions $\psi^L$ and $\psi^U$. Everywhere below it is assumed that condition \textbf{A}   holds true.

Clearly, the functions $L$ and $U$, involved in the definition of $\psi^L$  and $\psi^U$,  respectively,  satisfy
$$
U(x_2)-U(x_1)=\int_{x_1}^{x_2} {2\over x}L(x)\, dx, \quad x_1<x_2.
$$
Then, by the Fubini theorem, we have the following relation for   $\psi^L$ and $\psi^U$:
\begin{equation}
\psi^U(\xi_2)-\psi^U(\xi_1) =\int_{\xi_1}^{\xi_2} \tfrac{2}{\eta} \psi^L(\eta)\,d\eta, \quad \xi_1<\xi_2.\label{psipm2}
\end{equation}
Combining this relation with  condition \textbf{A}, we obtain the lower growth bound for the real part of the characteristic exponent $\psi$.

\begin{lemma}\label{C1} Suppose that condition \textbf{A} holds true. Then there exists a constant $c>0$ such that
    \begin{equation}
    \RE\psi(\xi) \geq c |\xi|^{2/\beta}, \quad |\xi|\geq 1. \label{psi-al}
    \end{equation}
\end{lemma}

\begin{remark} By  the standard properties of the Fourier transform, it follows immediately from (\ref{psi-al})  that for every $t>0$ the distribution  density $p_t$ of $Z_t$ exists and   $p_t\in C_b^\infty(\real)$.
\end{remark}

\begin{proof}[Proof of the lemma.] Define
    \begin{equation}
    \theta^U (v):=\psi^U (e^v), \quad \theta^L (v):=\psi^L (e^v), \quad v\in \Re.    \label{theta}
    \end{equation}
It follows from \eqref{psipm2} that these functions are related as follows:
    \begin{equation}
    \theta^U(v_2) -\theta^U(v_1)=2\int_{v_1}^{v_2} \theta^L (v)dv, \quad v_1<v_2. \label{theta2}
    \end{equation}
Combined with condition \textbf{A}, this gives
\begin{equation}
    \theta^U(v_2) -\theta^U(v_1)\geq {2\over \beta}\int_{v_1}^{v_2} \theta^U (v)dv, \quad v_1<v_2. \label{theta21}
    \end{equation}
This means that for the Stieltjes measure on $\Re^+$, generated by the non-decreasing function $\theta^U$, the absolutely continuous part in its Lebesgue decomposition is bounded from below by $(2/\beta)\theta^U(u)\, du$. Using this observation, one can easily  show the inequality
 \begin{equation}
    e^{-(2/\beta)v_2}\theta^U(v_2) \geq e^{-(2/\beta)v_1}\theta^U(v_1), \quad v_1<v_2.\label{psi+0}
    \end{equation}
It follows from (\ref{mu}) and the definitions of the functions $\psi^U$ and $\theta^U$ that
$\theta^U(0)=\psi^U(1)>0$.
Therefore by (\ref{psi+0}) we get
$\theta^U(v)\geq c_1 e^{(2/\beta)v}$,  $v\geq 0$,
with $c_1=\theta^U(0)>0$, which in turn implies
$$
    \psi^U(\xi)\geq c_1|\xi|^{2/\beta}, \quad |\xi|\geq 1.
$$
Using condition \textbf{A} and the left inequality in (\ref{psipm1}), we get (\ref{psi-al}) with $c=c_1(1-\cos 1)\beta^{-1}$.

\noindent Together with the function $\rho_t$ defined by (\ref{rho}), we  consider the functions
\be
\rho_t^U :=\inf\{ \xi> 0:  \psi^U(\xi)\geq 1/t\}, \quad \rho_t^L :=\inf\{ \xi> 0:  \psi^L(\xi)\geq 1/t\}.\label{rhoUL}
\ee
In view of \eqref{psi-al},  \eqref{psipm1} and condition \textbf{A}, the sets in the above definition are non-empty.   Note that
\be\label{id}
\RE\psi(\rho_t)=\psi^U(\rho_t^U)=\psi^L(\rho_t^L)={1\over t},
\ee
because the functions $\psi$ and $\psi^U$ are continuous, and
the function $\psi^L$ is right continuous on $[0,\infty)$. In the sequel, we write $f_t\asymp g_t$, $t\in A$, if there  exist positive constants $c_1$, $c_2$ such that
$c_1f_t\leq g_t\leq c_2 f_t$, $t\in A$.

\end{proof}

\begin{lemma}\label{l-aux1}  For every  $c\in (0,\infty)$ one has
    $$
   \rho_t\asymp \rho_{ct}\asymp\rho_t^U\asymp\rho_t^L,\quad t\in(0,1].
    $$
\end{lemma}
\begin{proof}   From the relation
 $\psi^L\leq \psi^U$  and condition \textbf{A}  we have
$\rho^U_t\leq \rho_t^L\leq \rho^U_{\beta t}$.
In addition, by \eqref{psipm1} we have
$\rho^U_{2t}\leq \rho_t\leq \rho^L_{(1-\cos 1)t}$.
Therefore,  it is enough to prove that
    \begin{equation}
     \rho_{ct}^U\asymp \rho_t^U, \quad t\in(0,1]. \label{33}
    \end{equation}
Note that $\rho_t^U$ is monotone  by construction. Hence one of the bounds in \eqref{33} is trivial (the lower one if $c>1$, and the upper one if $c<1$).  The other bound  in  \eqref{33} follows from the inequality
$$
{\psi^U(\xi_2)\over \psi^U(\xi_1)}\geq \left(\xi_2\over \xi_1\right)^{2/\beta}, \quad \xi_2>\xi_1>0,
$$
which holds true due to \eqref{psi+0}.
\end{proof}

\subsection{Upper bounds}
Consider the decomposition (\ref{decomp}). Note that the real part of the  characteristic exponent $\psi_t$ of the variable $\bar Z_t$  satisfies
\be\label{Repsit}\ba
\mathrm{Re}\, \psi_t(\xi)&=t\mathrm{Re}\, \psi(\xi)-t\int_{|\rho_t u|\geq 1}\Big(1-\cos(\xi u)\Big)\mu(du)\\&\geq t\mathrm{Re}\, \psi(\xi)-2t\int_{|\rho_t u|\geq 1}\mu(du)\geq t\mathrm{Re}\, \psi(\xi)-2t\psi^U(\rho_t).\ea
\ee
Hence, by Lemma \ref{C1} for every $t>0$ the variable $\bar Z_t$ has a distribution density $\bar p_t\in C_b^\infty(\Re)$, and
\be\label{rep}
\frac{\partial^k }{\partial x^k}\bar p_t(x)={1\over 2\pi}\int_\Re (-i\xi)^ke^{-ix\xi-\psi_t(\xi)}\, d\xi, \quad k\geq 0.
\ee
By \eqref{conv1}, the upper bounds in Theorem \ref{t-main}, Theorem \ref{t-main1}, and Theorem \ref{aux-theo} are provided by the following lemma, which gives respectively the upper bounds for the  distribution density $\bar p_t$ and its derivatives.

\begin{lemma}\label{ptbar}
For any $k\geq 0$, $t_0>0$,  there exist positive constants $b_1$ and $b_2$ such that
\begin{equation}
\Big| \frac{\partial^k }{\partial x^k} \bar p_t (x+a_t)\Big| \leq b_1 \rho_t^{k+1} e^{-b_2\rho_t|x|}, \quad x\in \real, \quad t\in (0,t_0]. \label{ptbar-eq}
\end{equation}

\end{lemma}
\begin{proof}
We apply the complex analysis technique, similar to the one used in \cite{KS10a} and \cite{KK11}.

\noindent Observe that the function $\psi_t(\xi)$ has an analytic extension to the complex plane, and for given $x$ and $t$ the  function
$$
H(t,x,z)= -ixz-\psi_t(z)=-ixz-t\int_{|\rho_t u|<1} (1-e^{i uz}+i uz)\mu(du)
$$
is well defined for $z\in \comp$. We consider the integral in the right hand side of \eqref{rep} as the integral of the function
$(-iz)^k e^{H(t,x,z)}$, $z\in \comp$,
along the real line in $\comp$, and change the integration contour as follows.

 Fix $\eta\in \Re$ such that  $|\eta|\leq \rho_t$, and for $M>0$ consider the rectangular contour $\Gamma^M=\bigcup_{j=1}^4\Gamma^M_j$ with
$$
\Gamma_1^M=\{y+i 0, \quad |y|\leq M\}, \quad \Gamma_2^M=\{y+i\eta, \quad |y|\leq M\}, \quad \Gamma_{3,4}^M=\{\pm M+is\eta, s\in [0,1]\},
$$
properly oriented.   By the Cauchy theorem we have
\begin{equation}
\left[\int_{\Gamma^M_1}+\int_{\Gamma^M_2}+\int_{\Gamma^M_3}+\int_{\Gamma^M_4}\right](-iz)^ke^{H(t,x,z)}dz=\int_{\Gamma^M}(-iz)^ke^{H(t,x,z)}dz=0. \label{Cauchy}
\end{equation}
Let us show that the integrals along $\Gamma_j^M$, $j=3,4$, vanish as  $M\to \infty$. Consider the case  $j=3$, the case $j=4$ is completely analogous. We have
\begin{equation}
\begin{split}
\mathrm{Re} H(t,x,M+is\eta)&= sx\eta -t\int_{|\rho_t u|<1}\Big(1-us\eta-e^{-us\eta}\Big)\mu(du)\\&- t \int_{|\rho_t u|<1} e^{-s\eta u} (1-\cos (Mu))\mu(du)\\&
\leq sx\eta -\psi_t (is\eta)- e^{-1} \mathrm{Re} \psi_t(M). \label{H}
\end{split}
\end{equation}
Denote
$$
\Psi_t(\eta):=\psi_t(i\eta),
\quad C(t,x, \eta):=\sup_{s\in [0,1]}\Big(sx\eta - \Psi_t(s\eta)\Big).
$$
Then by \eqref{H} we have
$$
\left|\int_{\Gamma^M_3}(-iz)^ke^{H(t,x,z)}dz\right|\leq \Big(\eta^2+M^2\Big)^{k/2}\exp\Big(C(t,x, \eta)-e^{-1} \mathrm{Re} \psi_t(M)\Big).
$$
By Lemma \ref{C1} and (\ref{Repsit}), the above inequality implies that the integral along $\Gamma_3^M$ vanishes as $M\to \infty$.

\noindent Passing to the limit in (\ref{Cauchy}) as $M\to \infty$, we get the representation
$$
\frac{\partial^k}{\partial x^k} \bar p_t(x)= {1\over 2\pi}\int_{\Re+i\eta}(-iz)^ke^{H(t,x,z)}dz={1\over 2\pi}\int_{\mathbb{R}} (-iy+\eta)^k e^{\eta x-ixy-\psi_t(y+i\eta)}\, dy.
$$
Thus, by \eqref{H} with $y$ instead of $M$ and $s=1$, we  have
\begin{equation}
\Big|\frac{\partial^k}{\partial x^k} \bar p_t(x)\Big| \leq {1\over 2\pi} e^{\eta x-\Psi_t(\eta)} \int_\real \big( |\eta|+|y|\big)^k e^{-e^{-1}\RE\psi_t(y)}dy.\label{est1}
\end{equation}
Recall that $|\eta|\leq \rho_t$, and denote
$c:=\sup_{|x|\leq 1}{1-x-e^{-x}\over x^2}$.
Then, using the first inequality in \eqref{psipm1} and the first identity in \eqref{id}, we get  for $|\eta|\leq \rho_t$
$$
-\Psi_t(\eta)\leq c t\int_{|\rho_t u |<1} (\eta u)^2\mu(du)\leq ct\psi^L(\rho_t)\leq \left({c\over 1-\cos 1}\right)t\RE\psi(\rho_t)={c\over 1-\cos 1}.
$$
Hence, we can take  $\eta=-\rho_t\,\mathrm{sign}\, x$, and obtain from \eqref{est1}
$$
\Big|\frac{\partial^k}{\partial x^k} \bar p_t(x)\Big|\leq  c_1 e^{-\rho_t |x|} \int_\real \big( \rho_t^k+|y|^k\big) e^{-e^{-1}\RE\psi_t(y)}dy
$$
with
$c_1={2^{k-2}\over \pi} \exp\left({c\over 1-\cos 1}\right)$.

Recall that $\RE\psi_t$ satisfies \eqref{Repsit} and note that, by  \eqref{psipm1} and condition \textbf{A}, the term
$t\psi^U(\rho_t)$, $t\in (0,1]$,
involved in the right hand side of \eqref{Repsit}, is bounded. In addition, we have by \eqref{psipm1} and condition~\textbf{A} the inequality
$\RE\psi(\xi)\geq {1-\cos 1\over \beta}\psi^U(\xi)$.
Using these observations   we finally get  the estimate
$$
\Big|\frac{\partial^k}{\partial x^k} \bar p_t(x)\Big|\leq  c_2 e^{-\rho_t |x|}\left(\rho_t^k I_0\left(t, {1-\cos 1\over \beta e}\right)+I_k\left(t, {1-\cos 1\over \beta e}\right)\right),
$$
where
$$
I_k(t,\lambda):=\int_\real |y|^k e^{-\lambda t \psi^U(y)}dy,\quad k\geq 0,
$$
and  $c_2>0$ is some suitable constant.  The rest of the proof follows from Lemma \ref{l-aux2} below.

\end{proof}

\begin{lemma}\label{l-aux2} For every $k\geq 0$, $\lambda>0$, there exists a positive constant $c=c(k,\lambda)$ such that
$$
I_k(t, \lambda)\leq c \rho_t^{k+1},  \quad t\in (0,1].
$$
\end{lemma}
\begin{proof} Clearly, the function $\rho_t$, $t\in (0,1]$, is bounded from  below by some positive constant. Since the function
$$
I_k^0(t,\lambda):=\int_{-1}^1 |y|^k e^{-\lambda t \psi^U(y)}dy,\quad t\in (0,1],
$$
is bounded, it remains to prove
$$
I_k^+(t, \lambda)\leq c \rho_t^{k+1},\quad I_k^-(t, \lambda)\leq c\rho_t^{k+1} \quad t\in (0,1],
$$
where
$$
I_k^+(t,\lambda):=\int_1^{\infty} y^k e^{-\lambda t \psi^U(y)}dy,\quad I_k^-(t,\lambda):=\int_{-\infty}^{-1} (-y)^k e^{-\lambda t \psi^U(y)}dy.
$$
We prove the required relation for $I_k^+(t,\lambda)$, the case of $I_k^-(t,\lambda)$ is completely analogous.

\noindent Making the change of variables $y=e^v$, we get
$$
I_k^+(t, \lambda)=\int_0^\infty e^{(k+1)v-t\lambda \theta^U(v)}\,dv,
$$
see the proof of Lemma \ref{C1} for the definition of $\theta^U$ and other auxiliary functions. Take $v_t=\log \rho_t^L$. Since  $\theta^U$ is non-negative, we have
\begin{align*}
I_k^+(t, \lambda)& = (\rho_t^L)^{k+1}\int_0^\infty e^{(k+1)(v-v_t)-t\lambda (\theta^U(v)-\theta^U(v_t)) -t\lambda\theta^U(v_t)}\, dv.
\\&
\leq (\rho_t^L)^{k+1}\int_0^\infty e^{(k+1)(v-v_t)-t\lambda (\theta^U(v)-\theta^U(v_t))}\, dv.
\end{align*}
The function $\theta^U$ is non-decreasing, hence
$$
\int_0^{v_t} e^{(k+1)(v-v_t)-t\lambda (\theta^U(v)-\theta^U(v_t))}\, dv\leq \int_0^{v_t} e^{(k+1)(v-v_t)}\,dv\leq {1\over k+1}.
$$
On the other hand, we have from the last identity in (\ref{id}) and the trivial relation $\psi^L\leq \psi^U$ that $t\theta^U(v_t)\geq 1$.
Then, using the monotonicity of   $\theta^U$, condition \textbf{A}, and
 applying twice \eqref{theta2},  we get
    \begin{align*}
    t[\theta^U(v)-\theta^U({v_t})]&=2t\int_{v_t}^v\theta^L(r)dr
    \\&
    \geq 2t\beta^{-1}\int_{v_t}^v\theta^U(r)dr
    =
    2t\beta^{-1}\theta^U({v_t})(v-v_t)+4t\beta^{-1} \int_{v_t}^v\int_{v_t}^r\theta^L(s)\,ds dr
    \\&
    \geq 2t\beta^{-1}\theta^U({v_t})(v-v_t)+ 4t\beta^{-2}\int_{v_t}^v\int_{v_t}^r\theta^U({s})\,ds\, dr
      \\&
    \geq 2\beta^{-1}(v-v_t)+ 2\beta^{-2} (v-v_t)^2.
    \end{align*}
Therefore,
$$
\int_{v_t}^\infty e^{(k+1)(v-v_t)-t\lambda \theta^U(v)+t\lambda\theta^U(v_t)}\, dv\leq \int_0^\infty e^{(k+1)w-(2\lambda/\beta)w- (2\lambda/\beta^2)w^2}\, dw<\infty.
$$
Combining the above inequalities and using Lemma \ref{l-aux1}, we obtain  the required statement.
\end{proof}

In the symmetric case  we get essentially the analogue of Lemma~\ref{ptbar}, but with a slightly different upper bound. Note that in this case we have $a_t=x_t=0$.
\begin{lemma}\label{ptbar2}
Let $\psi$ be real valued. Then for any $k\geq 0$, $t_0>0$,  there exist $b_1, b_2>0$ such that
\begin{equation}
\Big| \frac{\partial^k }{\partial x^k} \bar p_t (x)\Big| \leq b_1 \rho_t^{k+1} e^{-b_2\rho_t|x|\ln (\rho_t |x|+1)}, \quad x\in \real, \quad t\in (0,t_0]. \label{ptbar-eq-sym}
\end{equation}
\end{lemma}
\begin{proof}
The proof is essentially the same as in the non-symmetric case. First, we change the integration contour  in the same way as  in the proof of Lemma~\ref{ptbar}.  Unlike the non-symmetric case, while doing this we  do not pose  any additional assumptions on  $\eta\in \real$. The functions $H(t,x,z)$,  $\Psi_t(\eta)$ and  $C(t,x,\eta)$ have the same meaning as before, we only need to keep in mind that now $\psi_t(\xi)$ admits  the representation
$$
\psi_t(\xi)= t  \int_{|\rho_t u|\leq 1} (1-\cos (\xi u)) \mu(du).
$$
We have
\begin{equation}
\begin{split}
\RE  H(t,x,M+i\eta)&= x\eta+ t\int_{|\rho_tu|\leq 1}
(\cosh (\eta u)-1)\mu(du)- t\int_{|\rho_t u| \leq 1} \cosh (\eta u)[1-\cos (yu)]\mu(du)
\\&
\leq x\eta- \psi_t(i\eta) -\psi_t(M).
\end{split}
\end{equation}
Thus,
$$
\left| \int_{\Gamma_i^M} (-iz)^k e^{H(t,x,z)} dz\right| \leq
(\eta^2 +M^2)^{k/2} e^{C(t,x,\eta)-t\psi_t(M)},\quad i=3,4.
$$
Therefore, applying the Cauchy theorem to the same contour as in the proof of Lemma~\ref{ptbar} and letting $M\to\infty$, we derive
\begin{equation}
\Big|\frac{\partial^k }{\partial x^k} \bar p_t(x)\Big|\leq  \frac{1}{2\pi} e^{\eta x - \psi_t(i\eta)}  \int_\real (|y|+|\eta|)^ke^{-t\psi_t(y)}dy.
\end{equation}
By Lemma~\ref{l-aux2} we get
$$
\int_\real (|\eta|+|y|)^k e^{-\psi_t(y)} dy\leq c_1 (|\eta|^k\rho_t + \rho_t^{k+1} ),
$$
where $c_1>0$ is some constant, depending on $k$. Since the L\'evy measure $\mu$ is symmetric, we have
\begin{align*}
-\psi_t(i\eta)&= t\int_{|\rho_t u|<1}[\cosh(\eta u)-1]\mu(du)=t\int_{|\rho_t u|<1}(\eta u)^2\vartheta(\eta u)\mu(du)
\\&
\leq t\vartheta(\eta/\rho_t)\int_{|\rho_t u|<1}(\eta u)^2\mu(du)
=t(\eta/\rho_t)^2
\vartheta(\eta/\rho_t)\psi^L(\rho_t)
\\&
=\cosh(\eta/\rho_t)-1,
\end{align*}
where $\vartheta(x)=x^{-2}[\cosh x-1]$, and we used that $\vartheta$ is even and strictly increasing on $(0,\infty)$.

\noindent Thus,  by Lemma~\ref{l-aux2}, we obtain
\begin{equation}
\begin{split}
\Big|\frac{\partial^k}{\partial x^k} \overline{p}_t(x)\Big|&\leq c_1 e^{x\eta +  \cosh(\eta /\rho_t)} (|\eta|^k \rho_t + \rho_t^{k+1})
\\&
\leq c_2 \rho_t^{k+1}e^{x\eta +  \cosh(\eta /\rho_t)+k \ln (\eta/\rho_t)}
\\&
\leq c_2 \rho_t^{k+1}e^{x\eta + c_3 \cosh(\eta /\rho_t)}.
\label{dpdx2}
\end{split}
\end{equation}
The upper bound follows by minimizing the right-hand side in $\eta$ over all $\eta\in \real$.
\end{proof}

\subsection{Lower bounds}  To obtain the lower bound, observe  that the upper bound for $\bar p_t(x)$
implies the existence of $L>0$ such that
$$
\int_{|\rho_t x|\leq L}\bar p_t(x)\, dx\geq 1/2.
$$
Then
\begin{equation}
1/2 \leq  \int_{|\rho_t x|\leq L}p_t(x)\, dx \leq \tfrac{2L}{\rho_t}\max_{x\in \real} \bar  p_t(x), \label{low1}
\end{equation}
which proves the lower estimate in  Theorem~\ref{t-main}. Using additionally  that
$\big|{\prt\over \prt x} \bar p_t(x)\big| \leq b \rho_t^2$ with some $b>0$ by Lemma \ref{ptbar}, we obtain by the Taylor expansion argument
$$
\bar p_t (x) \geq \bar p_t(x_t) - \Big|\int_{x_t}^x \frac{\partial}{\partial x} \bar p_t(y)dy\Big|\geq \tfrac{1}{4L} \rho_t - c_2 \rho_t^2.
$$
\noindent  This together with \eqref{conv1}  implies the statement II  of Theorem~\ref{t-main1}, because condition \textbf{A} provides that the family $\Lambda_t(\Re)$, $t\in(0,1]$, is bounded.

\subsection{Proof of Theorem~\ref{Klup}}

I.    Put
$$
H(du):= e^{-C} \sum_{n=1}^\infty \frac{C^n G^{*n}(du)}{n!}.
$$
By \eqref{dens02} we can rewrite the upper estimate for $p_t(x)$ as
\begin{equation}
p_t(x)\leq \rho_t f_{upper}(x\rho_t) + c_1 \rho_t \int_1^\infty
\left(f_{upper}(x\rho_t -u) +f_{upper}(x\rho_t +u)\right) H(du). \label{dens20}
\end{equation}
Since $G$ is by our assumption sub-exponential, we have  by Theorem~3 from \cite{EGV79} that the distribution function $H$ is sub-exponential as well, and
$$
\lim_{x\to \infty} \frac{1-G(x)}{1-H(x)}=C,
$$
which implies for $x$ large enough
$c_2 (1-H(x))\leq 1-G(x)\leq c_3(1-H(x))$,
where $c_2,c_3$ are some positive constants. Then for any positive function $f$ vanishing at $\infty$ we have
\begin{align*}
\int_1^\infty f(u)H(du)& = \int_1^\infty (1-H(u))df(u) +f(1)
\asymp \int_1^\infty (1-G(u))df(u) +f(1)
\\&
=\int_1^\infty f(u) G(du).
\end{align*}
Thus, we can write
\begin{equation}
p_t(x)\leq \rho_t f_{upper}(x\rho_t) + c_4 \rho_t \int_1^\infty
\left(f_{upper}(x\rho_t -u) +f_{upper}(x\rho_t +u)\right) G(du). \label{dens201}
\end{equation}
Let us estimate the right-hand side for  $x>0$, the case $x<0$ is completely analogous. For $x>0$ and $v>0$ we have $f_{upper}(x+v) \leq f_{upper}(x)$, which gives
$$
\int_1^\infty f_{upper}(x+v)dG(v) \leq f_{upper}(x).
$$
Integration by parts gives us for the other term in \eqref{dens201}
\begin{align*}
\int_1^\infty f_{upper}(x-v)dG(v) &= -\int_1^\infty G(v) df_{upper}(x-v)
\\&
\leq \int_1^\infty (1-G(v)) |f_{upper}'(x-v)| dv+ f_{upper}(x-1)
\\&
= \int_{-\infty}^{x-1}(1-G(x-v))|f_{upper}'(v)|dv+ f_{upper}(x-1),
\end{align*}
(here we understand the derivative in the sense of the derivative of an absolutely continuous function). Since $G$ is sub-exponential and $|f_{upper}'|$ is integrable,
$$
\lim_{x\to \infty }  \int_{-\infty}^{x-1} \frac{1-G(x-v)}{1-G(x)}|f_{upper}'(v)|dv = \int_\real |f_{upper}'(v)|dv.
$$
Thus, for $x$ large enough we have
$$
\int_1^\infty (1-G(v)) |f_{upper}'(x-v)| dv\leq c_5 (1-G(x)).
$$
Finally, by the comment after Definition~\ref{sub} we get \eqref{bell1}.

II.  Denote
$$
h_t(u):= \sum_{n=1}^\infty\frac{m_t^{*n}(u)}{n!}.
$$
By \eqref{dens01},
$h_t(u) \leq h(u)$,
where
$h(u) := e^{-C} \sum_{n=1}^\infty\frac{C^n g^{*n}(u)}{n!}$
 is sub-exponential as well, and
\begin{equation}
\lim_{x\to \infty} \frac{h(y)}{g(y)} =C.\label{dens21}
\end{equation}
Therefore,
\begin{equation}
p_t(x)\leq \rho_t f_{upper}(x\rho_t)+ c_1\rho_t \int_1^\infty \left(f_{upper}(x\rho_t -u) +f_{upper}(x\rho_t +u)\right) g(y)dy.\label{dens10}
\end{equation}
By the comment after Definition~\ref{sub}, the function $g$ decays as $x\to \infty$ slower than any exponential function, implying
$$
 \int_1^\infty\left( f_{upper}(x\rho_t -y) +f_{upper}(x\rho_t +y)\right) g(y)dy \leq c_2 g(|x\rho_t|).
 $$
Thus, we can estimate the right-hand side of \eqref{dens10} as
$$
c_3\rho_t \left( f_{upper}(x\rho_t)+ g(|x\rho_t|)\right),
$$
 which finally gives \eqref{bell2}.
\qed

\section{Examples}\label{exa}
\begin{example}\label{exa1}

 Let $Z_t$ be a symmetric $\alpha$-stable process, $\alpha\in (0,2)$. In this case $\psi(\xi)=|\xi|^\alpha$, $\mu(du)=C_\alpha |u|^{-\alpha-1}du$, and one can easily verify that condition \textbf{A} is satisfied.

 In this case $\rho_t=t^{-1/\alpha}$. Because the L\'evy measure possesses the density $m(y)=C_\alpha |y|^{-1-\alpha}$,  we are in the situation of Theorem~\ref{Klup} II with  $g(u)=\alpha^{-1}u^{-1-\alpha}\I_{u\geq1}$,  which is obviously a sub-exponential density. Applying this theorem,  we arrive at the well-known upper estimate  for  the symmetric $\alpha$-stable distribution density:
 $$
 p_t(x)\leq c_1 t^{-1/\alpha} \left(e^{-b_2 t^{-1/\alpha}|x|\ln (t^{-1/\alpha}|x|+1)}+ \Big(t^{-1/\alpha}|x|\Big)^{-\alpha-1} \I_{t^{-1/\alpha}|x|\geq 1} \right) \leq c_2 t^{-1/\alpha}\wedge \frac{t}{|x|^{1+\alpha}}.
$$
In addition, it is straightforward to verify that the lower compound kernel estimate from Theorem~\ref{t-main1}  now provides the similar  lower bound
$$
p_t(x)\geq c_3 t^{-1/\alpha}\wedge \frac{t}{|x|^{1+\alpha}}
$$
with some positive constant $c_3$. In other words, using our main Theorem \ref{t-main1} we can re-establish the well-known two-sided estimate for
the symmetric $\alpha$-stable distribution density:
\begin{equation}\label{eq-exa1}
p_t(x)\asymp t^{-1/\alpha}f(t^{-1/\alpha}x), \quad (t, x)\in (0,\infty)\times \Re,
\end{equation}
with
\begin{equation}\label{eq-exa11}
f(x)=1\wedge |x|^{-\alpha-1}.
\end{equation}
\end{example}
\begin{example}\label{exa2}  \emph{(a)} Consider  a L\'evy process with the discrete L\'evy measure
$$
\mu(dy)=\sum_{n=-\infty}^\infty 2^{n\gamma}\Big(\delta_{2^{-n\upsilon}}(dy)+\delta_{-2^{-n\upsilon}}(dy)\Big),
$$
where  $\upsilon>0$, $0<\gamma<2\upsilon$. Straightforward calculation shows that in this case \eqref{orey} holds true with $\alpha=\gamma/\upsilon$. In particular, condition \textbf{A} is satisfied  and $\rho_t \asymp t^{-1/\alpha}$. This means that there exists $q>0$ such that for every $t\in (0,1]$ the inequality
$n\leq  n_0(t):=\left[{1\over \gamma}\log_2{1\over t}-q\right]$
implies
$2^{-n\upsilon}\rho_t>1$.
Consequently,
$$
\Lambda_t(dy)\geq t\sum_{n\leq n_0(t) } 2^{n\gamma}\Big(\delta_{2^{-n\upsilon}}(dy)+\delta_{-2^{-n\upsilon}}(dy)\Big),
$$
and, taking into account that in this case $a_t=x_t=0$,   one has by Theorem~\ref{t-main1}.II that for some positive constant $c$
$$
p_t(2^{-n\upsilon})\geq ct\rho_t 2^{n\gamma}, \quad n\leq n_0(t).
$$
Then every function $f_t(x)$, monotonous with respect to  $x$ on $\real^+$ and satisfying \eqref{bell}, should satisfy
$$
f_t(x)\geq ct\rho_t 2^{n\gamma}, \quad x\in (2^{-n\upsilon}, 2^{-(n-1)\upsilon}], \quad  n\leq n_0(t).
$$
It is easy to verify that when  $\gamma\leq \upsilon$ (which is equivalent to $\alpha\leq 1$) one has
$$
t^{1-1/\alpha}\sum_{n\leq n_0(t) } 2^{n\gamma} 2^{-n\upsilon}\to \infty, \quad t\to 0+.
$$
This gives $\int_{\Re} f_t(x)dx\to \infty$ as $t\to 0+$,  which  shows that in this case  any estimate of the form \eqref{bell} with a ``bell-like function'' $f_t(x)$ would be  extremely inexact.

\emph{(b)} Consider now the case $1<\alpha<2$.  Note that for $x>1$
\begin{align*}
 t\mu\Big(\{u: |\rho_t u|>x\}\Big) =2t \sum_{n \leq n(t,x)} 2^{\gamma n} \leq  Ct 2^{\tfrac{\gamma}{\upsilon} \log_2 (\rho_t/x)}
 = C x^{-\gamma/\upsilon}= Cx^{-\alpha},
\end{align*}
 where $n(t,x):= \tfrac{1}{\upsilon}\log_2 (\rho_t/|x|)$. Therefore, condition \eqref{dens02} of Theorem~\ref{Klup}.I holds true with $1-G(x)=x^{-\alpha}$, $x\geq 1$. By this theorem and Remark~\ref{rem1}, we have the following estimate for the respective transition probability density:
\begin{equation}
p_t(x) \leq c_1 t^{-1/\alpha}\Big(b_1 e^{-b_2 t^{-1/\alpha} |x|\ln (t^{-1/\alpha} |x|+1)}+   \Big(t^{-1/\alpha}|x|\Big)^{-\alpha} \I_{\{t^{-1/\alpha}|x|\geq 1\}} \Big)\leq c_2 t^{-1/\alpha}f(t^{-1/\alpha}x)\label{eq-exa2}
\end{equation}
with
\begin{equation}
f(x)=1\wedge |x|^{-\alpha}.\label{eq-exa21}
\end{equation}
This upper bound looks very similar the one from (\ref{eq-exa1}), with the notable difference in the functions  $f$ in (\ref{eq-exa11}) and (\ref{eq-exa21}). This difference appears because in the previous example we have used the version of Theorem \ref{Klup} based on sub-exponential densities, while in the current example we had to use another version of this theorem which deals with sub-exponential distribution  functions. Note that (\ref{eq-exa2}) is precise in the sense that
\begin{equation}
p_t(x)\geq c_2 t^{-1/\alpha}f(t^{-1/\alpha}x), \quad t>0, \quad |x|\in \{2^{-n\upsilon}, n\geq 1\}\label{eq-exa3}
\end{equation}
with $f$ given by (\ref{eq-exa21}). The latter inequality can be  verified easily using the lower bound in Theorem \ref{t-main1}. On the other hand, one can hardly expect that (\ref{eq-exa3}) can be extended to hold true for all $x\in \Re$, which would be  a complete analogue of the lower bound in \eqref{eq-exa1}.

 \end{example}

 The third example shows that condition \textbf{A} is comparatively mild. In particular, it  allows some kind of ``slow oscillations''  for the characteristic exponent.
\begin{example}\label{exa3}
Consider the function $\alpha: [0,\infty)\to [\alpha_-,\alpha_+]\subset (0,2)$ such that
    \begin{equation}
    v\alpha'(v)\to 0\quad \text{ as}\quad  v\to \infty. \label{v1}
    \end{equation}
Put
   \begin{equation}
    \theta^-(v):=\int_0^v e^{\alpha(w)w}dw, \quad \theta^+(v):=2\int_0^v \theta^-(w)dw.\label{v11}
    \end{equation}
Using  the l'Hospital rule  we get by \eqref{v1}
    \begin{equation}
    \theta^-(v)\sim \frac{e^{\alpha(v)v}}{\alpha(v)}\quad \text{as}\quad v\to\infty. \label{v2}
    \end{equation}
Then the function $\phi(\xi):=\theta^-(0\vee \ln \xi)$ satisfies

a)  $\frac{d}{d\xi}\left( \frac{\phi(\xi)}{\xi^2}\right)\leq 0$  (which follows from \eqref{v11} and \eqref{v2});

b)  $\frac{\phi(\xi)}{\xi^2}\to 0$ as $\xi\to\infty$.

Thus, there exists  a L\'evy measure $\mu(du)=m(u)du$  such that
    $$
    \phi(\xi)=\int_{|\xi u|\leq 1} (\xi u)^2 m(u)du,
    $$
or, in other words, $\phi(\xi)\equiv \psi^L(\xi)$, where $\psi^L(\xi)$ is the lower function, corresponding to $\mu$. Then the corresponding upper function is
 $\psi^U(\xi)\equiv  \theta^+(0\vee \ln \xi)$. Similarly to \eqref{v2} we get
    \begin{equation}
    \theta^+(v)\sim \frac{2e^{\alpha(v)v}}{\alpha^2(v)},\quad v\to\infty, \label{v3}
    \end{equation}
which implies together with \eqref{v2}
    $$
    \frac{\psi^U(\xi)}{\psi^L(\xi)}\sim \frac{2}{\alpha (\ln \xi)}, \quad \xi\to\infty.
    $$
Thus, $\psi$ satisfies condition \textbf{A}, and from \eqref{v2} and \eqref{v3} we get
    \begin{equation}
    \psi^U(\xi)\asymp \psi(\xi)\asymp \psi^L(\xi)\asymp |\xi|^{\alpha(\ln |\xi|)}, \quad |\xi|\to\infty.
    \end{equation}
In particular, for any $\alpha\in [\alpha_-,\alpha_+]$   there exists a sequence $\{\xi_{\alpha,n}$, $n\geq 1\}$, such that
    $$
    \xi_{\alpha, n}\to \infty, \quad \psi(\xi_{\alpha,n})\asymp (\xi_{\alpha,n})^\alpha, \quad n\to\infty.
    $$
Thus, by Theorem~\ref{t-main} for any $\alpha\in [\alpha_-,\alpha_+]$ there exists  a sequence $t_{\alpha,n}\to 0$  such that
    $$
    p_{t_{\alpha,n}}(0)\asymp (t_{\alpha,n})^{-1/\alpha}, \quad n\to\infty.
    $$
In this example, heuristically,    the asymptotic behaviour of the characteristic exponent $\psi$ may coincide with the one  of the characteristic exponents of $\alpha$-stable processes on various subsets of $\real^+$ having $\infty$ as their limit point. Here the index $\alpha$ depends on the subset, and may vary in the range $[\alpha_-,\alpha_+]$. Respectively, the asymptotic behaviour of the transition probability density as $t\to0+$ depends heavily on the particular subset of $\Re^+$  this behavior is considered on, and may be the same as for the  $\alpha$-stable processes with any $\alpha$  from the segment $[\alpha_-,\alpha_+]$.
\end{example}

\textbf{Acknowledgement.} The authors thank N.Jacob, R. Schilling  and A. Bendikov for their helpful comments and remarks. The first-names author gratefully acknowledges the Scholarship of NAS of Ukraine for young scientists (2009-2011), and the Scholarship of the President of Ukraine for young scientists (2011-2013).

\end{document}